\documentclass[final,onefignum,onetabnum]{siamart220329}
\usepackage{mymacros}
\usepackage{textcomp,geometry,graphicx,enumerate,fullpage}
\usepackage{booktabs}
\usepackage{setspace}
\usepackage{hyperref}
\usepackage{mathrsfs} 
\usepackage{cleveref}
\usepackage{todonotes}

\newsiamthm{example}{Example}
\newsiamthm{remark}{Remark}

\title{Low-Rank Univariate Sum of Squares Has No Spurious Local Minima}

\author{Beno\^it Legat \and Chenyang Yuan \and Pablo A. Parrilo}

\date{\today}

\begin{document}
\maketitle

\begin{abstract}
  We study the problem of decomposing a polynomial $p$ into a sum of $r$ squares
  by minimizing a quadratically penalized objective
  $f_p(\vecp{u}) = \norm{\sum_{i=1}^r u_i^2 - p}^2$. This objective is nonconvex
  and is equivalent to the rank-$r$ Burer--Monteiro factorization of a
  semidefinite program (SDP) encoding the sum of squares decomposition. We show
  that for all univariate polynomials $p$, if $r \ge 2$ then $f_p(\vecp{u})$ has
  no spurious second-order critical points, showing that all local optima are
  also global optima. This is in contrast to previous work showing that for
  general SDPs, in addition to genericity conditions, $r$ has to be roughly the
  square root of the number of constraints (the degree of $p$) for there to be
  no spurious second-order critical points. Our proof uses tools from
  computational algebraic geometry and can be interpreted as constructing a
  certificate using the first- and second-order necessary conditions. We also
  show that by choosing a norm based on sampling equally-spaced points on the
  circle, the gradient $\nabla f_p$ can be computed in nearly linear time using
  fast Fourier transforms. Experimentally we demonstrate that this method has
  very fast convergence using first-order optimization algorithms such as
  \mbox{L-BFGS}, with near-linear scaling to million-degree polynomials.
\end{abstract}

\begin{keywords}
nonconvex optimization, sum of squares, Burer--Monteiro method, first-order methods, trigonometric polynomials, global landscape, semidefinite programming
\end{keywords}

\begin{MSCcodes}
90C23, 90C26, 90C22
\end{MSCcodes}

\section{Introduction}
Burer--Monteiro factorization \cite{Burernonlinearprogrammingalgorithm2003} is a
methodology to solve large-scale semidefinite programs (SDPs) by replacing
positive semidefinite (PSD) variables $X \succeq 0$ with a factorization
$X = UU^\top$. This automatically enforces the PSD constraint and lets us find
low-rank solutions by choosing the rank of the new variable $U$. In addition,
this factorization results in a nonlinear optimization problem that can be
solved with first-order methods with fast per-iteration times especially when
$\rank(U)$ is small. However, the resulting problem is nonconvex so these
methods may get stuck in local optima. We show that this will not happen to the
SDP finding the sum of squares decomposition of univariate polynomials; in this
setting all local optima are also global.

In this work we study SDPs arising from sum of squares optimization
\cite{ParriloStructuredsemidefiniteprograms2000}. The ability to represent the
cone of sum of squares polynomials as a SDP enables many applications in
polynomial optimization, control, and relaxations of combinatorial problems
\cite{laurent2009sums, blekherman_semidefinite_2013}. To determine if a
polynomial $p(x) \in \Rx[2d]$ is a sum of squares, it suffices to find a
feasible solution to the following SDP:
\begin{align} \label{eq:sos-sdp}
  \arraycolsep=1.5pt
  \begin{array}{rcl}
    p(x) &=& b(x)^\top X b(x) \\
    X &\succeq& 0
  \end{array}
\end{align}
where $b(x)$ is a suitable polynomial basis of $\Rx[d]$. The constraint
$p(x) = b(x)^\top X b(x)$ defines an affine subspace of the space of symmetric
matrices, which can be expressed by matching the coefficients of $p(x)$ with the
corresponding coefficients of the polynomial $b(x)^\top X b(x)$.

Given the factorization $X = UU^\top$ where $\bar{u}_1, \ldots, \bar{u}_r$ are
the column vectors of $U$, we have the explicit sum of squares decomposition
\begin{align} \label{eq:sos-decomp}
  \textstyle
  p(x) = \sum_{i=1}^r u_i(x)^2,
\end{align}
where $u_i(x) = b(x)^\top \bar{u}_i$. In other words, \eqref{eq:sos-decomp} is a
formulation for the Burer--Monteiro factorization of \eqref{eq:sos-sdp}
independent of any particular basis. Instead of solving a SDP to find $u_i(x)$,
we apply a quadratic penalty to the equality constraint \eqref{eq:sos-decomp} to
arrive at the following nonconvex objective:
\begin{align} \label{eq:nonconvex-obj}
  f_p(\vecp{u}) = \norm{\textstyle \sum_{i=1}^r u_i(x)^2 -  p(x)}^2,
\end{align}
where $\vecp{u} = \bmat{u_1 & \cdots & u_r}$ is a vector of $r$ degree-$d$
polynomials and the norm is induced by any inner product on polynomials of
degree-$2d$. Then a degree-$2d$ polynomial $p(x)$ is a sum of $r$ squares if and
only if
\begin{align*}
  \min_{\vecp{u} \in \Rx[d]^r} f_p(\vecp{u}) = 0.
\end{align*}

We say that $\vecp{u}$ is a \emph{first-order critical point} (FOCP) of
$f_p(\vecp{u})$ when $\nabla f_p(\vecp{u}) = 0$, where the derivatives are taken
with respect to the variables in $\vecp{u}$. Since for every $p$ we can find
spurious FOCPs (for example taking $\vecp{u} = 0$), it is essential to consider
second-order necessary conditions. If $\vecp{u}$ also satisfies
$\nabla^2 f_p(\vecp{u}) \succeq 0$, then it is a \emph{second-order critical
  point} (SOCP) of $f_p(\vecp{u})$.

Every local minimum of a function is a SOCP, but the converse is not true when
the function is nonconvex\footnote{Consider, for example, $f(x) = x^3$ at
  $x=0$.}. If $\vecp{u}$ is not a SOCP, then we can produce a descent direction
using the gradient or Hessian. This leads to efficient first-order methods
\cite{ge2015escaping, jin2017escape} that converge to SOCPs. Thus if we can show
that at every SOCP $f_p(\vecp{u}) = 0$, these algorithms will always converge to
a global minimum. A recent line of work \cite{boumal2020deterministic,
  bhojanapalli2018smoothed} has shown that for general SDPs under smoothed
analysis or genericity conditions, when the rank of the factorization is above
the Barvinok--Pataki bound (roughly the square root of the number of
constraints), there are no spurious SOCPs.  Moreover, the smoothed analysis or
genericity conditions are necessary, as \cite{bhojanapalli2018smoothed}
constructed a SDP where only a full-rank factorization can guarantee no spurious
SOCPs.

In this paper we consider the setting of univariate polynomial optimization,
which is a class of problems with applications in signal processing, control
\cite{roh2006discrete, dumitrescu2007positive}, and computing equilibria
of polynomial games \cite{parrilo2006polynomial}. These optimization problems
involving nonnegative univariate polynomials can be transformed (i.e., by a
bisection on the objective value) into feasibility problems for finding the sum
of squares decomposition of a univariate polynomial. The main result of our
paper shows that without any additional assumptions, the rank-2
quadratic-penalized Burer--Monteiro factorization of the SDP describing the sum
of squares decomposition \eqref{eq:nonconvex-obj} of a univariate polynomial has
no spurious SOCPs.
\begin{theorem} \label{thm:main} For all nonnegative univariate polynomials
  $p(x)\in \Rx[2d]$ and any $r \ge 2$, if $\vecp{u} \in \Rx[d]^r$ satisfies
  $\nabla f_p(\vecp{u}) = 0$ and $\nabla^2 f_p(\vecp{u}) \succeq 0$, then
  $f_p(\vecp{u}) = 0$.
\end{theorem}
In particular, the rank bound in our result matches the Pythagoras number for
univariate polynomials \cite[Example 2.13]{choi1995sums}\footnote{The Pythagoras
  number for $\Sx[2d]$ is the smallest $r$ such that all polynomials in
  $\Sx[2d]$ can be written as a sum of $r$ squares of polynomials in
  $\Rx[d]$.}. In comparison, applying rank bounds for general SDPs to this
setting require $r \gtrsim \sqrt{d}$ (see \Cref{sec:background} for more
details).

\Cref{thm:main} is proved in \Cref{sec:main-proof}, by deriving a series of
increasingly stronger sufficient conditions (\eqref{eq:first_order_prop},
\eqref{eq:second_order_prop} and \eqref{eq:prop_factored}) implying
$f_p(\vecp{u}) = 0$ for increasingly larger classes of $\vecp{u}= (u_1, u_2)$,
eventually proving the result for all $\vecp{u} \in \R[x]^2$. To illustrate
this, we first show that when $r=2$ and $u_1, u_2$ are coprime,
$\nabla f_p(\vecp{u}) = 0$ implies that $f_p(\vecp{u}) = 0$ (the precise
statement and proof of this case are in \Cref{sec:g=1_h=1}). Note that in this
simplified setting only the first-order gradient condition is needed. By
computing the gradient, $\nabla f_p(\vecp{u}) = 0$ is equivalent to
\begin{align*}
  \nabla f_p(\vecp{u})(\vecp{v}) = \dotp{u_1 v_1 + u_2v_2, u_1^2 + u_2^2 - p} = 0
\end{align*}
for all $\vecp{v} = (v_1, v_2) \in \R[x]^2$. Since $u_1, u_2$ are coprime,
B\'ezout's identity (\Cref{lem:coprime_kernel}) implies that we can find
$\vecp{v}' = (v_1', v_2')$ so that
\begin{align} \label{bezout-simple}
  u_1 v_1' + u_2 v_2' = u_1^2 + u_2^2 - p,
\end{align}
thus showing that $f_p(\vecp{u}) = \norm{u_1^2 + u_2^2 - p}^2 = 0$. However, we
cannot assume a priori that $u_1, u_2$ are coprime as we are only given $p$ as
the input. The main technical contribution of our proof is how to handle the
more involved case when $u_1, u_2$ share a common factor.

We can also interpret our proof as a certificate. When we choose $\vecp{v}'$
satisfying \eqref{bezout-simple}, we obtain the identity
\begin{align*}
  \nabla f_p(\vecp{u})(\vecp{v}') = f_p(\vecp{u}).
\end{align*}
This implies that $f_p(\vecp{u}) = 0$ when $\nabla f_p(\vecp{u}) = 0$. In
\Cref{sec:cert} we generalize this example to our full proof of \Cref{thm:main},
showing for all $\vecp{u}$ and $p$ how to find $\vecp{v}'$ and $Q \succeq 0$
satisfying the following identity:
\begin{align*}
  \nabla f_p(\vecp{u})(\vecp{v}') + \dotp{Q, \nabla^2 f_p(\vecp{u})}= -f_p(\vecp{u}).
\end{align*}
From this identity it is clear that if $\vecp{u}$ is a SOCP, then
$f_p(\vecp{u}) = 0$. This compact form of our proof allows us to easily extend
our result to other problems in \Cref{sec:extensions}.

Since \Cref{thm:main} holds for any inner product, we can choose one that
enables efficient computation of $\nabla f_p(\vecp{u})$. When $p$ is a
degree-$2d$ univariate polynomial, an equivalent way of ensuring the constraint
$p(x) = b(x)^\top X b(x)$ in \eqref{eq:sos-sdp} is to write $2d+1$ constraints
$p(\hat{x}_i) = b(\hat{x}_i)^\top X b(\hat{x}_i)$, where
$\hat{x}_0, \ldots, \hat{x}_d$ are distinct sample points. This formulation can
be cast into the following least-squares objective,
\begin{align} \label{eq:main-opt}
  \min_{U \in \R^{(2d+1) \times r}} f_p(U)
  = \frac{1}{2d+1} \sum_{i=1}^{2d+1} \paren{\norm{U^\top \bi}_2^2 - \pxi}^2.
\end{align}
This is equivalent to choosing an inner product in \eqref{eq:nonconvex-obj} that
evaluates the polynomial on $2d+1$ points. If we choose $2d+1$ points on the
complex unit circle, we can compute $\nabla f_p(U)$ in $O(d \log d)$ time using
the Fast Fourier Transform (FFT). In \Cref{sec:experiments}, we show that this
method exhibits linear convergence experimentally using unconstrained
optimization algorithms such as \mbox{L-BFGS}, and has near-linear scaling to
million-degree polynomials.

\subsection{Contributions}
In summary, our main contributions in this paper are:
\begin{enumerate}
\item Proving that the quadratic penalty form of Burer--Monteiro factorization
  for univariate polynomial sum of squares decomposition has no spurious SOCPs
  (\Cref{thm:main}). Our result holds where the rank of the factorization is at
  least 2, matching the Pythagoras number for univariate polynomials. This is in
  contrast to previous work requiring rank $r \gtrsim \sqrt{d}$ in addition to
  genericity conditions or smoothed analysis for general SDPs, or showing that
  no spurious local minima exist in statistical problems.
\item Developing a new framework for proving that there are no spurious SOCPs
  for a quadratic-penalized factorized SDP, by constructing a certificate
  \eqref{eq:cert-general} using the first- and second-order necessary
  conditions. This certificate representation helps us extend our results to
  projection onto the sum of squares cone (\Cref{cor:projection}), certifying
  nonnegativity on intervals (\Cref{cor:intervals}) and sum of squares
  optimization.
\item Showing that by choosing a special norm (based on the evaluation of the
  polynomial on points on the unit circle), the full gradient of the objective
  can be computed in near-linear time using FFTs. It enables us to efficiently
  scale first-order methods to instances with millions of variables
  (\Cref{table:timing}). This is possible because our result \eqref{thm:main} is
  independent of the penalty function.
\end{enumerate}

\section{Background and Related Work}
\label{sec:background}
Let $\Sbb_n$ be the space of $n \times n$ symmetric matrices.  Given
$A_i \in \Sbb_n$, we consider the standard-form semidefinite feasibility problem
with variable $X \in \Sbb_n$:
\begin{align} \label{prob:sdp} \tag{\textsc{SDP}}
  \arraycolsep=1.5pt
  \begin{array}{rcl}
    \dotp{A_i, X} &=& b_i \\
    X &\succeq& 0
  \end{array}
\end{align}

\paragraph{Nonconvex formulation}
  Burer and Monteiro \cite{Burernonlinearprogrammingalgorithm2003} introduced the
  nonconvex reformulation $X = UU^\top$ to enforce the semidefinite constraint,
  where $U \in \R^{n \times r}$:
  \begin{align}\label{prob:nsdp} \tag{\textsc{NSDP}\textsubscript{r}}
    \dotp{A_i, UU^\top} = b_i.
  \end{align}
  This motivates the following least-squares formulation:
  \begin{align} \label{prob:sdpls} \tag{\textsc{SDPLS}\textsubscript{r}}
    \begin{array}{ll}
      \displaystyle \min_{U \in \R^{n \times r}}\,& \sum_i \norm{\dotp{A_i, U U^\top} - b_i}^2.
    \end{array}
  \end{align}
  If in addition to the equality constraints in \eqref{prob:sdp} one wishes to
  minimize the objective $\dotp{C, X}$, the works
  \cite{BurerLocalMinimaConvergence2005, bhojanapalli2018smoothed,
    CifuentesPolynomialtimeguarantees2019} formulate an augmented Lagrangian
  problem, where the term $\lambda \dotp{C, UU^\top}$ is added to the objective
  of \eqref{prob:sdpls}.

\paragraph{General SDP rank bounds} Burer and Monteiro subsequently showed in
  \cite{BurerLocalMinimaConvergence2005} that when $r \ge n$, there are no
  spurious SOCPs to \eqref{prob:sdpls}. This result is in fact tight and
  \cite{bhojanapalli2018smoothed} constructed an explicit instance where if
  $r = n-1$, one can find a SOCP that is not a global minimum. Thus for general
  SDP feasibility, additional conditions on the objective or analysis must be
  imposed. Then the rank bound can be improved to the maximum rank of extreme
  points of the section of the PSD cone with $m$ affine constraints. This
  maximum rank is $O(\sqrt{m})$, also known as the Barvinok--Pataki bound
  (\cite{BarvinokRemarkRankPositive2001, PatakiRankExtremeMatrices1998}). In the
  same work, Burer and Monteiro \cite{BurerLocalMinimaConvergence2005} showed
  that when a linear objective is added to \eqref{prob:sdp} and \eqref{prob:nsdp},
  if $r \gtrsim \sqrt{m}$, any local minimum of \eqref{prob:nsdp} is also a local
  minimum of \eqref{prob:sdp} with an additional rank-$r$ constraint. Then they
  showed that such a local minimum is either an optimal extreme point, or
  contained within the relative interior of a face of the feasible set of
  \eqref{prob:sdp} which is constant with respect to the objective function.
  Subsequent work \cite{bhojanapalli2018smoothed} then showed that if $C$ is
  generic enough, all local minima of \eqref{prob:nsdp} are global minima of
  \eqref{prob:sdp} (see Cifuentes and Moitra
  \cite{CifuentesPolynomialtimeguarantees2019} for more references). In summary
  this line of work requires generic constraints, and in addition either
  smoothness of the constraint set (\cite{boumal2016nonconvex,
    boumal2020deterministic}) or smoothed analysis
  (\cite{bhojanapalli2018smoothed, CifuentesPolynomialtimeguarantees2019,
    cifuentes2021burer}). In addition, \cite{waldspurger2020rank} showed that
  when $r$ is smaller than $\sqrt{m}$, SOCPs are not generically optimal.

\paragraph{Structured SDPs} Problems such as matrix completion and matrix sensing can
  be expressed as instances of \eqref{prob:sdpls} There has been a lot of recent
  interest in studying the global landscape of matrix sensing problems
  (\cite{ge2016matrix,bhojanapalli2016global,ge2017no}). A recent line of work
  (\cite{ge2017no, ge2016matrix, bandeira2016low}) shows that for certain
  statistical problems aiming to recover a signal in the form of a low-rank
  matrix corrupted by noise (where the SDP has a rank-1 solution in the
  noiseless setting), there are no spurious SOCPs when the noise level is low
  enough. Similar results \cite{ge2017no,li2019non,sun2018geometric} can be
  obtained for matrix sensing, where a low-rank matrix is reconstructed from
  linear measurements called sensing operators. See
  \cite{ChiNonconvexOptimizationMeets2019} for a survey of these problems. In
  summary, for a wide range of statistical problems, local minima are also
  global minima. These results are either satisfied with high probability, or
  require that the sensing operators $A_i$ satisfy the Restricted Isometry
  Property (RIP).

\paragraph{Sampling basis} The sampling or interpolation basis for sum of squares
  optimization is studied in \cite{Lofbergcoefficientssamplesnew2004} and
  \cite{CifuentesSamplingalgebraicvarieties2017}. Informally, they showed that
  if the sampled points are ``generic'' enough, the problem expressed in the
  sampled basis is equivalent to the original problem. This idea is also used
  for univariate polynomial optimization in \cite{papp2017univariate,Kapelevich2022}.

\paragraph{Univariate polynomials} Univariate/trigonometric polynomial optimization
  and their applications are studied in \cite{wu1999fir, roh2006discrete,
    dumitrescu2007positive} and the references therein. The decomposition of a
  nonnegative trigonometric polynomial into a sum of squares is also known as
  its spectral factorization \cite[Theorem
  1.1]{dumitrescu2007positive}. Previous methods for spectral factorization
  either require finding all $n$ roots of the polynomial, solving linear systems
  of order $n$, or use an approximate $O(N \log N)$ FFT-based algorithm by
  sampling $N \gg n$ points \cite{wu1999fir,dumitrescu2007positive}. Design
  problems involving constraints on nonnegative trigonometric polynomials can
  be formulated as SDPs. Due to their special structure, \cite{roh2006discrete}
  used FFTs to speed up per-iteration complexity for interior point methods
  solving these SDPs to $O(n^3)$.
  The set of all $X$ satisfying \eqref{eq:sos-sdp} is known as the Gram
  spectrahedron of $p(x)$. As the bounds developed for general SDPs depend on
  the rank of extreme points of the Gram spectrahedra, one may wonder if this
  quantity can be tightly bounded (i.e., better than the Barvinok--Pataki bound)
  in the special case of univariate polynomials. A recent work by Scheiderer
  \cite{scheiderer2022extreme} showed that this is not possible. If $p(x)$ is a
  sufficiently general positive univariate polynomial of degree $d$, its Gram
  spectrahedron has extreme points of all ranks up to $O(\sqrt{d})$.

\subsection{Notation}
Let $\Rx[d]$ be the space of univariate polynomials of degree at most $d$. Let
$\vecpoly{u} \in \Rx[d]^r$ be a vector of $r$ polynomials where each $u_i(x)$ is
a polynomial in $\Rx[d]$. Let $\sig: \Rx[d]^r \rightarrow \Rx[2d]$ be the
quadratic map defined by $\vecpoly{v} \mapsto \sum_{i=1}^r v_i(x)^2$, and
$\Sx[2d] \defeq \cone(\sig(\Rx[d]^1)) \subseteq \Rx[2d]$ be the cone of sum of
squares univariate polynomials of degree-$2d$. A \emph{binary form} is a
homogeneous polynomial in two variables. Let $\Rxhom[d]$ be the space of binary
forms of degree-$d$ and $\Sxhom[2d] \subseteq \Rxhom[2d]$ be the space of sum of
squares binary forms of degree-$2d$. $\Rxhom[d]$ and $\Sxhom[2d]$ are isomorphic
to $\Rx[d]$ and $\Sx[2d]$ respectively.

\subsection{Univariate Polynomials}
Any monic univariate polynomial $p \in \Rx[d]$ can be uniquely factored as
$p(x) = \prod_{i=1}^d (x - \alpha_i)$, where $\alpha_i \in \C$ are the roots of
$p$. Given univariate polynomials $p$, $g$ and $q$, we define an equivalence
relation $p \equiv q \pmod{g}$ if there exists $w \in \Rx$ such that
$p = q + wg$. We say that $g$ is a \emph{divisor} of $p$ if
$p \equiv 0 \pmod{g}$. In addition, if $g$ is also a divisor of $q$, then we say
that $g$ is a \emph{common divisor} of $p$ and $q$. Let $\gcd(p,q)$ be the
\emph{greatest common divisor} of $p$ and $q$, which is the common divisor with
the highest degree. By the unique factorization of $p$ and $q$, $\gcd(p, q)$ is
unique up to multiplication by a scalar. We say that $p$ and $q$ are
\emph{coprime} if $\gcd(p,q) = 1$.

Any binary form $p \in \Rxhom[d]$ can be factored as
$p(x_1, x_2) = \prod_{i=1}^d (\alpha_i x_1 - \beta_i x_2)$, where
$(\alpha_i, \beta_i) \in \C^2$. This factorization is unique up to
multiplication of $(\alpha_i, \beta_i)$ by a scalar. The equivalence relation
$p \equiv q \pmod{g}$ and $\gcd$ on binary forms are defined analogously to
those on univariate polynomials.

\section{Preliminary Results}
Since the polynomials we are optimizing over have fixed degrees, to better keep
track of this degree we will work with homogeneous polynomials (forms). In
\Cref{sec:socp-conditions} we derive expressions for the first- and second-order
necessary conditions (\eqref{eq:gradA} and \eqref{eq:hessA}). Next in Sections
\ref{sec:binary-forms} and \ref{sec:hgroup-proof} we review some results on
binary forms and prove our main \Cref{lem:hgroup}. Particularly important is our
decomposition of $\vecp{u}$ in \Cref{prop:decomp}.

\subsection{First and Second Order Necessary Conditions}
\label{sec:socp-conditions}
First we derive expressions for the gradient and Hessian of $f_p(\vecp{u})$.  In
addition to the binary forms we consider in this paper, the derivations in this
section also hold for general multivariate forms. For any inner product on forms
$\dotp{\cdot, \cdot}: \Rxhom[2d] \times \Rxhom[2d] \rightarrow \R$ and its
associated norm $\norm{\cdot}: \Rxhom[2d] \rightarrow \R$, the objective
function is written as
\begin{align*}
  \textstyle
  f_p(\vecpoly{u}) = \norm{\sum_{j=1}^r u_j(x)^2 - p(x)}^2.
\end{align*}
To find the gradient and Hessian of the objective, we compute the first- and
second-order terms of $f_p(\vecpoly{u} + \epsilon \vecpoly{v})$ to obtain:
\begin{align}
  \label{eq:grad}
  \frac{1}{4}\nabla_{\vecp{u}} f_p(\vecpoly{u})(\vecpoly{v})
  &= \dotp{{\textstyle \sum_{j=1}^r u_j(x)v_j(x)}, \textstyle \sum_{j=1}^r u_j(x)^2 - p(x)}, \\
  \label{eq:hess}
  \frac{1}{4}\nabla_{\vecp{u}}^2 f_p(\vecpoly{u})(\vecpoly{v}, \vecpoly{v})
  &= \dotp{\textstyle \sum_{j=1}^r v_j(x)^2, \textstyle \sum_{j=1}^r u_j(x)^2 - p(x)}
    + 2\norm{\textstyle \sum_{j=1}^r u_j(x)v_j(x)}^2.
\end{align}
Given a vector of polynomials $\vecpoly{u} \in \Rxhom[d_1]^r$, we define the
linear map $\Au: \Rxhom[d_2]^r \to \Rxhom[d_1+d_2]$ as
\begin{align*}
  \Au: \vecpoly{v} \mapsto \sum_{i=1}^r u_i(x) v_i(x).
\end{align*}
For example, we can write $\sigma(\vecp{u}) = \Aup(\vecp{u})$. When $n=2$,
$\Aup$ is up to a constant the map induced by the \emph{Sylvester} matrix
\cite{sturmfels2002solving}. If $d_1 = d_2$, the determinant of $\Aup$
is up to a constant the \emph{resultant} of $u_1$ and $u_2$; it vanishes if and
only if $u_1$ and $u_2$ share a common divisor. Next we concisely define SOCPs
using this new notation.
\begin{definition}\label{def:socp}
  We say that $\vecp{u} \in \Rxhom[d]^r$ is a second-order critical point (SOCP)
  of $f_p(\vecp{u})$ if its gradient is zero and its Hessian is positive
  semidefinite. In other words, for all $\vecp{v} \in \Rxhom[d]^r$,
  \begin{align}
    \label{eq:gradA}
    \frac{1}{4}\nabla_{\vecp{u}} f_p(\vecp{u})(\vecp{v})
    &= \dotp{\Aup(\vecp{v}), \sig(\vecp{u}) - p} = 0, \\
    \label{eq:hessA}
    \frac{1}{4}\nabla_{\vecp{u}}^2 f_p(\vecp{u})(\vecp{v}, \vecp{v})
    &= \dotp{\sig(\vecp{v}), \sig(\vecp{u}) - p} + 2\norm{\Aup(\vecp{v})}^2 \ge 0.
  \end{align}
\end{definition}

\begin{remark} \label{remark:socp-necessary} The first-order condition
  \eqref{eq:gradA} alone is insufficient to guarantee global optimality, even
  when $p$ is generic. For example, $\vecp{u} = 0$ is always a FOCP, but is
  spurious if $p \ne 0$. Since we can always construct spurious FOCPs, even when
  $p$ is generic, we need to consider the second-order conditions.
\end{remark}

\subsection{Pairs of Binary Forms}
\label{sec:binary-forms}
In this section we present some results on pairs of binary forms that will be
helpful for characterizing the sets\footnote{For the more algebrically inclined
  reader, the sets $\imag(\Aup)$ and $\ker(\Aup)$ are the graded parts of
  the \emph{ideal}~\cite[Definition~1.4.1]{cox2015groebner} and \emph{syzygy
    module}~\cite[Definition~10.4.3]{cox2015groebner} of $\vecp{u}$ respectively.}
$\imag(\Aup)$ and $\cone\paren{\sig\paren{\ker(\Aup)}}$ and proving
\Cref{thm:main} in \Cref{sec:main-proof}. From now on we assume that $r=2$ and
$\vecp{u} = (u_1, u_2)$.

The following lemma is a restatement of B\'ezout's lemma for univariate
polynomials using our notation. It states that any form in $\Rxhom[2d]$ is in
the image of the map $\Aup$, as long as $u_1$ and $u_2$ are coprime. We provide
a proof below for completeness.
\begin{lemma}
  \label{lem:coprime_kernel}
  Given $\vecp{u} = (u_1, u_2) \in \Rxhom[d_1]^2$ and $d_2 \ge d_1 - 1$,
  consider the map $\Aup : \Rxhom[d_2]^2 \to \Rxhom[d_1+d_2]$. Then $u_1$ and
  $u_2$ are coprime if and only if $\imag(\Aup) = \Rxhom[d_1+d_2]$.
  \begin{proof}
    For the ``if'' direction, when $u_1$ and $u_2$ are not coprime, then they
    share a common factor $w$ with degree at least 1, and so does every form in
    $\imag(\Aup)$. Thus $\imag(\Aup)$ is not equal to $\Rxhom[d_1+d_2]$.

    Next we prove the ``only if'' direction. Since
    $\Aup(\vecp{v}) = v_1 u_1 + v_2 u_2$, for every $\vecp{v} \in \ker(\Aup)$ we
    have
    \begin{align} \label{eq:coprime_equiv}
      u_1v_1 = -u_2v_2.
    \end{align}
    If $d_2 = d_1 - 1$, since $u_1$ and $u_2$ are coprime, by evaluating
    \eqref{eq:coprime_equiv} on the roots of $u_1$ and $u_2$ we can conclude
    that $v_1 = v_2 = 0$ and $\dim(\ker(\Aup)) = 0$. Otherwise if $d_2 \ge d_1$,
    this implies that $u_1$ is a divisor of $v_2$ and $u_2$ is a divisor of
    $v_1$. So there exists $w_1, w_2 \in \Rxhom[d_2-d_1]$ such that
    $v_1 = w_1u_2, v_2 = w_2u_1$. Then \eqref{eq:coprime_equiv} becomes
    $(w_1 + w_2)u_1u_2 = 0$, implying that $w_2 = -w_1$.  Thus
    $\ker(\Aup) = \braces{(w u_2, -w u_1) \mid w \in \Rxhom[d_2-d_1]}$ and
    $\dim(\ker(\Aup)) = \dim(\Rxhom[d_2-d_1])$. In summary, $d_2 \ge d_1-1$
    implies that $\dim(\ker(\Aup)) = d_2 + 1 - d_1$. Therefore by the
    rank-nullity theorem,
    \begin{align*}
      \dim(\imag(\Aup))
      = 2\dim(\Rxhom[d_2]) - \dim(\ker(\Aup))
      = \dim(\Rxhom[d_1 + d_2]).
    \end{align*}
    Since $\imag(\Aup) \subseteq \Rxhom[d_1+d_2]$ and these two sets have the
    same dimensions, they must be equal.
  \end{proof}
\end{lemma}

In particular, \Cref{lem:coprime_kernel} motivates the decomposition
$u_1 = u_1' \gc$ and $u_2 = u_2' \gc$, where $u_1'$ and $u_2'$ are coprime and
$\gc = \gcd(u_1, u_2)$. Indeed when $\gc = 1$ we can apply the sufficient
condition \eqref{eq:first_order_prop} in \Cref{sec:g=1_h=1} to show that
$\vecp{u}$ is not a SOCP. When $\gcd(\gc, \sig(\vecp{u'})) = 1$ we can apply
\eqref{eq:second_order_prop} in \Cref{sec:h=1}. Otherwise we need to partition
the roots of $\gc = gh$ by whether each root is also a root of
$\sig(\vecp{u'})$, then apply \eqref{eq:prop_factored} in
\Cref{sec:general-proof}.

\begin{proposition} \label{prop:decomp} Given
  $\vecp{u} = (u_1, u_2) \in \Rxhom[d]^2$, we can always find $g \in \Rxhom[m]$,
  $h \in \Rxhom[k]$ and $\vecp{u'} = (u_1', u_2') \in \Rxhom[d-m-k]^2$ so that
  \begin{align*}
    (u_1, u_2) &= (u_1'gh, u_2'gh), \\
    \gcd(u_1, u_2) &= gh, \\
    \gcd(u_1', u_2') &= 1, \\
     \gcd(\sig(\vecp{u'}), g) = \gcd({u_1'}^2 + {u_2'}^2, g) &= 1, \\
    r \text{ is a (possibly complex) root of } h & \Rightarrow r \text{ is a root of } \sig(\vecp{u'}).
  \end{align*}
  Moreover, this decomposition is unique up to multiplication by constants.
  \begin{proof}
    Let $\gc = \gcd(u_1, u_2)$ so that $u_1' = u_1/\gc$ and $u_2' = u_2/\gc$. By
    partitioning the roots of $\gc$ we can decompose $\gc = gh$, where $g$ has no
    common roots with $\sig(\vecp{u'})$ and every root of $h$ is also a root of
    $\sig(\vecp{u'})$. The uniqueness of this decomposition follows from the
    unique factorization theorem for binary forms.
  \end{proof}
\end{proposition}

We demonstrate this decomposition with an example.
\begin{example} \label{ex:h_not_1} Let
  $u_1(x_1,x_2) = (x_1^2 + x_2^2)^2 (2x_1^2 + x_2^2) x_1^2$ and
  $u_2(x_1,x_2) = (x_1^2 + x_2^2)^2 (2x_1^2 + x_2^2) x_1 x_2$. Then the
  decomposition in \Cref{prop:decomp} gives $u_1' = x_1$, $u_2' = x_2$,
  $g = x_1 (2x_1^2 + x_2^2)$ and $h = (x_1^2 + x_2^2)^2$.
\end{example}

The following observation about the roots of $\sig(\vecp{u'})$ and $h$ is useful
in our proofs.
\begin{proposition} \label{prop:no-real-roots} In the decomposition of
  \Cref{prop:decomp}, both $\sig(\vecp{u'})$ and $h$ have no real roots and
  $\deg(h)$ is even.
  \begin{proof}
    If $\sig(\vecp{u'})$ has a real root $\hat{x}$, then
    $u_1'(\hat x)^2 + u_2'(\hat x)^2 = 0$. Thus $\hat{x}$ is a root of $u_1'$
    and $u_2'$, contradicting the fact that $\gcd(u_1', u_2') = 1$. Thus every
    root of $h$ is complex, and $h$ must have even degree.
  \end{proof}
\end{proposition}

\subsection{Main Lemma}
\label{sec:hgroup-proof}
In our proof of \Cref{thm:main} we will use following main result, which may be
of independent interest.
\begin{lemma}
  \label{lem:hgroup}
  Given binary forms $g \in \Rxhom[m]$, $q \in \Sxhom[2d-2m]$ and
  $p \in \Sxhom[2d]$, if $g$ and $q$ are coprime then there exists a sum of
  squares binary form $s \in \Sxhom[2m]$ such that
  \begin{equation*}
    p \equiv s q \pmod{g}.
  \end{equation*}
\end{lemma}
The main ingredient of the proof of \Cref{lem:hgroup} is a result stating that
any univariate polynomial $a(x)$ strictly positive on the real zeros of $g(x)$
can be written as a single square\footnote{The number of squares is not
  important in our proof of \Cref{lem:hgroup}, we only need the property that
  $a(x)$ can be written as a sum of squares modulo $g(x)$.}
modulo $g(x)$.
\begin{proposition}
  \label{prop:soscoprime}
  Let $g(x)$ and $a(x)$ be coprime univariate polynomials where $\deg(g) =
  m$. If $a(x) > 0$ for all $\{x \in \R \mid g(x) = 0 \}$ then there exists a
  polynomial $t \in \R[x]_{m}$ such that
  \begin{align*}
    a \equiv t^2 \pmod{g}.
  \end{align*}
\end{proposition}
This result is related to Schm{\"u}dgen's certificate \cite{schmudgen1991thek},
which states that if a polynomial is strictly positive on a compact
semialgebraic set, then it has a Positivstellensatz certificate in terms of the
equations describing the set. We prove \Cref{prop:soscoprime} using Hermite
interpolation on the series expansion of $\sqrt{a(x)}$ around the roots of $g$.
\begin{proof}[Proof of \Cref{prop:soscoprime}]
  Let $r_i$ be the roots (possibly complex) of $g(x)$,
  each with multiplicity $n_i$, so that $\sum_i n_i = m$. Consider the Taylor
  series expansion of $f(x) = \sqrt{a(x)}$ centered at $r_i$. Since $a$ and $g$
  do not share any common roots as they are coprime, this Taylor series is well
  defined around any root of $g$. Let the polynomials $\gamma_i(x)$ be the first
  $n_i$ terms of the Taylor expansion of $f(x)$ centered at $r_i$. The
  polynomials $\gamma_i$ have real coefficients if $r_i$ is real, and if $r_i$
  and $r_j$ are a pair of conjugate roots, $\gamma_i = \bar{\gamma}_j$. We can
  then use the Chinese Remainder Theorem \cite[Section 7.6]{dummit2003abstract}
  to construct the unique polynomial $t(x)$ with real coefficients and
  $\deg(t) < m$ such that
  \begin{align*}
    \forall i, \qquad t(x) \equiv \gamma_i(x) \pmod{(x - r_i)^{n_i}}.
  \end{align*}
  By construction, for all roots $r_i$ of $g$ and any $k = 0, \ldots, n_i-1$, we have
  \begin{align*}
    \frac{d^k}{dx^k} f(r_i) = \frac{d^k}{dx^k} \gamma_i(r_i) = \frac{d^k}{dx^k} t(r_i).
  \end{align*}
  For each root $r_i$, we have
  \begin{align*}
    \sqrt{a(r_i)} = f(r_i) = \gamma_i(r_i) = t(r_i).
  \end{align*}
  and
  \begin{align*}
    \frac{d}{dx} a(r_i) = 2 f(r_i) \frac{d}{dx} f(r_i) = 2 t(r_i) \frac{d}{dx} t(r_i)
    = \frac{d}{dx} t(r_i)^2.
  \end{align*}
  By induction we get $\frac{d^k}{dx^k}a(r_i) = \frac{d^k}{dx^k}t(r_i)^2$ for
  $k = 0, \ldots, n_i-1$. This is a generalization of Hermite interpolation
  for a variable number of consecutive derivatives at each point \cite[section
  17.6]{hamming1973numericalmethods}\footnote{This is referred to as Birkhoff
    interpolation in \cite{hamming1973numericalmethods}, and the existence of
    a unique interpolating polynomial crucially depends on the use of
    consecutive derivatives.}. Since $a(x)$ and $t(x)^2$ match at all the roots of $g$
  (including derivatives up to the multiplicity of the root), we have shown that $a \equiv t^2 \pmod{g}$.
\end{proof}

Then we prove the affine version of \Cref{lem:hgroup}.
\begin{lemma} \label{lem:group} Let $g, p$ and $q$ be univariate polynomials
  where $\deg(g) = m$, $\deg(p) = 2d$ and $\deg(q) = 2d-2m$, $p$ and $q$ are sum
  of squares, and $g$ and $q$ are coprime. Then there exists a sum of squares
  polynomial $s \in \Sx[2m]$ such that
  \begin{equation*}
    p \equiv s q \pmod{g}.
  \end{equation*}
\end{lemma}
\begin{proof}
	Since $q$ is coprime with $g$, \Cref{lem:coprime_kernel} (after reducing $q$ modulo
  $g$) guarantees that there exists a polynomial $a \in \Rx[m]$ such that
  \begin{equation}
    \label{eq:bezout}
    aq \equiv 1 \pmod{g}.
  \end{equation}
  We have $q(x) > 0$ for all real roots $x$ of $g$, since $q$ is nonnegative
  and coprime with $g$. Thus $a(x) > 0$ for all real roots $x$ of $g$, by
  evaluation of \eqref{eq:bezout} at these roots. Since $a(x)q(x) = 1$ for all
  roots $x$ of $g$, $a$ is also coprime with $g$. Then we can apply
  \Cref{prop:soscoprime} to find $t \in \Rx[m]$ so that $a \equiv t^2
  \pmod{g}$. Multiplying both sides of \eqref{eq:bezout} by $p$, we get
  \begin{align*}
    p \equiv t^2 p q \pmod{g}.
  \end{align*}
  Since $t^2p$ is a sum of squares, we can reduce each squared polynomial modulo
  $g$ to get $s \in \Sx[2m]$.
\end{proof}

Finally we prove \Cref{lem:hgroup}, which is the projective version of
\Cref{lem:group}.
\begin{proof}[Proof of \Cref{lem:hgroup}]
  We first apply a linear change of coordinates so that $(0, 1)$ is not a root
  of $g$, $p$, or $q$. Then let $g'(x) = g(x, 1)$, $p'(x) = p(x, 1)$ and
  $q'(x) = q(x, 1)$. Since this dehomogenization procedure preserves the degree
  of $g$, $p$, and $q$ \footnote{If the degree of $g$ is not preserved after
    dehomogenization, the degree of $t'$ after applying \Cref{lem:group} could
    be larger than $2d-m$. For example, if $d=m=2$, $g=x_1x_2$,
    $p = (2x_1^2+x_2^2)x_2^2$ and $q=1$, we get that $s' = (x^2+1)^2$ and
    $t'=-x^3$ after dehomogenizing and applying \Cref{lem:group}. This issue
    will not occur if the dehomogenization is degree-preserving.}, we can apply
  \Cref{lem:group} to find polynomials $s' \in \Sx[2m]$ and $t' \in \Rx[2d-m]$
  so that
  \begin{align*}
    p' = s' q' + t'g'.
  \end{align*}
  We can then homogenize by letting $s(x_1, x_2) = x_2^{2m} s'(x_1/x_2)$ and
  $t(x_1, x_2) = x_2^{2d-m} t'(x_1/x_2)$. Thus $s$ is also a sum of squares and
  \begin{align*}
    p = s q + tg.
  \end{align*}
\end{proof}

\section{Main Theorem and Proof}
\label{sec:main-proof}
In this section we prove \Cref{thm:main}, which states that for univariate
polynomials, a rank-2 decomposition has no spurious second-order critical
points. Using the decomposition in \Cref{prop:decomp}, we first prove simplified
versions of \Cref{thm:main} in Sections \ref{sec:g=1_h=1} and \ref{sec:h=1},
before proving the full version in \Cref{sec:general-proof}.

\subsection{Coprime Case: $g=1$, $h=1$}
\label{sec:g=1_h=1}
This is the case explained in the introduction.
In the decomposition of \Cref{prop:decomp}, $g = h = 1$ implies that $u_1$ and
$u_2$ are coprime. This happens generically and implies that for a fixed $p$,
the gradient condition \eqref{eq:gradA} is sufficient for almost all $\vecp{u}$.
\begin{proposition} \label{prop:first_order_property}
  Suppose $\vecp{u} \in \Rxhom[d]^2$ and $p \in \Sxhom[2d]$ satisfies
  $\nabla f_p(\vecp{u}) = 0$. If
  \begin{align} \tag{C1}\label{eq:first_order_prop}
    p \in \imag(\Aup),
  \end{align}
  then we have $f_p(\vecp{u}) = 0$.
\end{proposition}
\begin{proof}
  Since $p \in \imag(\Aup)$, we can find $\vecp{v} \in \Rxhom[d]^2$ so that
  $\Aup(\vecp{v}) = \sig(\vecp{u}) - p$. Evaluating the gradient condition
  \eqref{eq:gradA} at $\vecp{v}$, we conclude that
  $f_p(\vecp{u}) = \norm{\sig(\vecp{u}) - p}^2 = 0$.
\end{proof}
Since \Cref{lem:coprime_kernel} implies that $\imag(\Aup) = \Rxhom[2d]$ if and
only if $u_1$ and $u_2$ are coprime, we have shown that when $g = h = 1$, we
always have $p \in \imag(\Aup)$ and there are no spurious FOCPs and SOCPs.

\subsection{Special Case: $h=1$}
\label{sec:h=1}
When $u_1$ and $u_2$ are not coprime, $\sigma(\vecp{u}) - p $ might not be in
$\imag(\Aup)$ and we cannot use the argument in
\Cref{prop:first_order_property}. Thus we need to use make use of the Hessian
condition \eqref{eq:hessA}.

\begin{proposition} \label{prop:second_order_property} Suppose
  $\vecp{u} \in \Rxhom[d]^2$ and $p \in \Sxhom[2d]$ satisfies
  $\nabla f_p(\vecp{u}) = 0$ and $\nabla^2 f_p(\vecp{u}) \succeq 0$. If
  \begin{align} \tag{C2}\label{eq:second_order_prop}
    p \in \imag(\Aup) + \cone\paren{\sig\paren{\ker(\Aup)}},
  \end{align}
  then we have $f_p(\vecp{u}) = 0$.
\end{proposition}

This means that if for all $\vecp{u}$ we can decompose $p = q + r$ where
$q \in \imag(\Aup)$ and $r \in \cone\paren{\sig\paren{\ker(\Aup)}}$, then
$f_p(\vecp{u})$ has no spurious SOCPs. In particular, similar to how
$\imag(\Aup)$ is related to the gradient condition \eqref{eq:gradA} in
\Cref{prop:first_order_property}, $\cone\paren{\sig\paren{\ker(\Aup)}}$ is
related to the Hessian condition \eqref{eq:hessA}. The following result states
that \eqref{eq:second_order_prop} is satisfied if $h=1$.

\begin{lemma} \label{lem:imag_plus_cone} Given
  $\vecp{u} \in \Rxhom[d]^2$, if in the decomposition of
  \Cref{prop:decomp} we have $h=1$ then for every $p \in \Sxhom[2d]$,
  \begin{align*}
    p \in \imag(\Aup) + \cone\paren{\sig\paren{\ker(\Aup)}}.
  \end{align*}
\end{lemma}
\begin{proof}
  We want to show that any
  $p(x) \in \Sxhom[2d]$ can be written as the sum of polynomials in
  $\imag(\Aup)$ and $\sig(\ker(\Aup))$. Therefore it is useful to have a
  characterization of these sets. \Cref{lem:coprime_kernel} tells us
  that
  \begin{align*}
    \imag(\Aup) &= \braces{\Aupp(\vecp{v}) g \mid \vecp{v} \in \Rxhom[d] } \\
                &= \braces{w g \mid w \in \Rxhom[2d-m]}.
  \end{align*}
  Since for all $t \in \Rxhom[m]$ we have
  $(-t u_2', t u_1') \in \ker(\Aup)$,
  \begin{align}
    \notag
    \braces{t^2 \sig(\vecp{u'}) \mid t \in \Rxhom[m]} &\subseteq \sig(\ker(\Aup)) \\
    \label{eq:conesigker}
    \braces{s \sig(\vecp{u'}) \mid s \in \Sxhom[2m]} &\subseteq \cone(\sig(\ker(\Aup))).
  \end{align}
  Since $\sig(\vecp{u'})$ is coprime with $g$ by assuming $h=1$, we can apply
  \Cref{lem:hgroup} to show that there exists $w \in \Rxhom[2d-m]$ and
  $s \in \Sxhom[2m]$ such that $p = s \sig(\vecp{u'}) + wg$.
\end{proof}

Finally we prove \Cref{prop:second_order_property}.
\begin{proof}[Proof of \Cref{prop:second_order_property}]
  The condition \eqref{eq:second_order_prop} implies that there exist
  $\vecp{v} \in \Rxhom[d]^2, \vecp{w}^{(i)} \in \ker(\Aup)$ such that
  \begin{align*}
    p = \Aup(\vecp{v}) + \sum_i \sig(\vecp{w}^{(i)}).
  \end{align*}
  Since $\nabla f_p(\vecp{u}) = 0$,
  \eqref{eq:gradA} implies that
  \begin{equation}
    \label{eq:second_order_property_grad}
    \dotp{\Aup(\vecp{v}), \sig(\vecp{u}) - p} = 0.
  \end{equation}
  Since $\nabla^2 f_p(\vecp{u}) \succeq 0$ and $\vecp{w}^{(i)} \in \ker(\Aup)$,
  \eqref{eq:hessA} implies that
  \begin{equation}
    \label{eq:second_order_property_hess}
    \dotp{\sig(\vecp{w}^{(i)}), \sig(\vecp{u}) - p} \ge 0.
  \end{equation}
  Combining \eqref{eq:second_order_property_grad} and \eqref{eq:second_order_property_hess} gives
  \begin{align} \label{eq:pup}
    \dotp{p, \sig(\vecp{u}) - p} \ge 0.
  \end{align}
  Since $\nabla f_p(\vecp{u}) = 0$ implies that
  \begin{align*}
    \dotp{\Aup(\vecp{u}), \sig(\vecp{u}) - p} =
    \dotp{\sig(\vecp{u}), \sig(\vecp{u}) - p} = 0,
  \end{align*}
  we have
  \begin{align*}
    f_p(\vecp{u}) = \norm{\sig(\vecp{u}) - p}^2
    = \dotp{\sig(\vecp{u}) - p, \sig(\vecp{u}) - p}
    = - \dotp{p, \sig(\vecp{u}) - p}.
  \end{align*}
  This together with \eqref{eq:pup} implies that $f_p(\vecp{u}) \le
  0$. However $f_p(\vecp{u})$ is always nonnegative, thus it must be 0.
\end{proof}

\subsection{General Case}
\label{sec:general-proof}
\Cref{lem:imag_plus_cone} alone is insufficient to prove \Cref{thm:main}. It is
possible for $\gc = \gcd(u_1, u_2)$ to share complex roots with $\sig(\vecp{u'})$
(recall from \Cref{prop:no-real-roots} that all roots of $\sig(\vecp{u'})$ are
complex), as seen in \Cref{ex:h_not_1}. Hence the argument in the proof of
\Cref{lem:imag_plus_cone} fails as $\gc$ is not coprime with
$\sig(\vecp{u'})$. To get around this issue, we will derive the sufficient
condition \eqref{eq:prop_factored} in \Cref{prop:second_order_property_factor},
a stronger version of \eqref{eq:second_order_prop}, by carefully examining the
Hessian condition \eqref{eq:hessA}. Roughly speaking,
\Cref{prop:second_order_property_factor} shows that we can replace every root of
$h$ (which must be complex) with any real root\footnote{Without loss of
  generality we choose this real root to be $x_1$.}. Since $\sig(\vecp{u'})$ has
no real roots, $\gc x_1^k/h$ is now coprime with $\sig(\vecp{u'})$, and we can
then complete the proof by following the argument in the previous section.

\begin{proposition} \label{prop:second_order_property_factor} Suppose
  $\vecp{u} \in \Rxhom[d]^2$ and $p \in \Sxhom[2d]$ satisfies
  $\nabla f_p(\vecp{u}) = 0$ and $\nabla^2 f_p(\vecp{u}) \succeq 0$, with the
  decomposition in \Cref{prop:decomp} where $k = \deg(h)$ and
  $\vecp{u}x_1^k/h = (u_1'g x_1^k, u_2'g x_1^k)$. If
  \begin{align} \tag{C3}\label{eq:prop_factored}
    p \in \imag\paren{\mathcal{A}_{\vecp{u}x_1^k/h}}
    + \cone\paren{\Sigma\paren{\ker(\Aup)}},
  \end{align}
  then $f_p(\vecp{u}) = 0$.
  \begin{proof}
    We first prove that if $r \in \Rxhom[\pmb{\ell}]$ is a common divisor of
    $\sig(\vecp{u'})$ and $\gc = \gcd(u_1, u_2)$, then
    \begin{align} \label{eq:identity-factored}
      \dotp{\Aup(\vecp{b}) x_1^\ell/r, \sig(\vecp{u}) - p} = 0,\,
      \text{for all } \vecp{b} = (b_1, b_2) \in \Rxhom[d]^2.
    \end{align}
    Given any $b_1, b_2 \in \Rxhom[d]$ and $\eta \in \R$, let
    $v_1 = \eta x_1^\ell u_2/r + b_2$ and $v_2 = -\eta x_1^\ell u_1/r - b_1$. We have
    \begin{align}\label{eq:Au-identity-factored}
      \Aup(\vecp{v}) & = u_1 (\eta x_1^\ell u_2/r + b_2) - u_2 (\eta x_1^\ell u_1/r + b_1)
                       = u_1b_2 - u_2b_1\\
      \label{eq:sig-identity-factored}
      \sig(\vecp{v}) & = \eta^2 x_1^{2\ell} (u_1^2 + u_2^2)/r^2
                       + 2\eta x_1^\ell (b_1 u_1 + b_2 u_2)/r
                       + (b_1^2 + b_2^2).
    \end{align}
    Since $r$ is a divisor of both $\sig(\vecp{u'})$ and $\gc$,
    $x_1^{2\ell} (u_1^2 + u_2^2)/r^2 = x_1^{2\ell} \sig(\vecp{u'}) \gc^2/r^2$ is
    a multiple of $\gc$. Thus $x_1^{2\ell} (u_1^2 + u_2^2)/r^2 \in \imag(\Aup)$
    and we have
    \begin{align*}
      \dotp{\eta^2 x_1^{2\ell} (u_1^2 + u_2^2)/r^2, \sig(\vecp{u}) - p} = 0.
    \end{align*}
    Therefore, the Hessian condition \eqref{eq:hessA} implies that for all
    $\eta \in \R$,
    \begin{equation}
      \label{eq:hesseta}
      2\eta \dotp{(b_1 u_1 + b_2 u_2)x_1^\ell/r, \sig(\vecp{u}) - p}
      + \dotp{b_1^2 + b_2^2, \sig(\vecp{u}) - p}
      + 2\norm{u_1b_2 - u_2b_1}^2 \ge 0.
    \end{equation}
    This implies the identity \eqref{eq:identity-factored}; otherwise there
    exists $\eta$ such that \eqref{eq:hesseta} is negative.

    Since $\sig(\vecp{u'})$ and $h$ have no real roots
    (\Cref{prop:no-real-roots}), we can write $h = \prod_{i=1}^{k/2} r_i$, where
    each $r_i \in \Rxhom[2]$ is a quadratic form corresponding to the product of
    a pair of complex roots. We first apply the same argument from above to show
    that
    \begin{align} \label{eq:prop_r1}
      \dotp{\Aup(\vecp{b}) x_1^2/r_1, \sig(\vecp{u}) - p} = 0,\,
      \text{for all } \vecp{b} \in \Rxhom[d]^2.
    \end{align}
    Next we show that \eqref{eq:prop_r1} implies that
    \begin{align} \label{eq:prop_r1r2}
      \dotp{\Aup(\vecp{b}) x_1^4/(r_1r_2), \sig(\vecp{u}) - p} = 0,\,
      \text{for all } \vecp{b} \in \Rxhom[d]^2.
    \end{align}
    Similar to before, let $r = r_1r_2$ so we have the identities
    \eqref{eq:Au-identity-factored} and \eqref{eq:sig-identity-factored} as
    before. Since $r_2$ is a divisor of both $\sig(\vecp{u'})$ and $\gc/r_1$,
    $x_1^{8} \sig(\vecp{u'}) \gc^2/(r_1r_2)^2 = x_1^{8}
    \frac{\sig(\vecp{u'})}{r_2} \frac{\gc}{r_1r_2} \frac{\gc}{r_1} $ is a
    multiple of $\gc/r_1$. Thus
    $x_1^{8} \sig(\vecp{u'}) \gc^2/(r_1r_2)^2 \in
    \imag\paren{\mathcal{A}_{\vecp{u}x_1^2/r_1}}$ and we then use
    \eqref{eq:prop_r1} to show \eqref{eq:prop_r1r2}.

    Thus by iteratively applying the previous arguments \footnote{The argument here
      is subtle because although every root of $h$ is a root of
      $\sig(\vecp{u'})$, a root may have higher multiplicity in $h$ than in
      $\sig(\vecp{u'})$. For example, it is possible that
      $h = (x_1^2 + x_2^2)^2$ but $\sig(\vecp{u'}) = x_1^2 + x_2^2$. In this
      case, to obtain \eqref{eq:identity-factored} we need to iteratively ``peel
      off'' the factors $r_1 = r_2 = x_1^2 + x_2^2$, by first proving
      \eqref{eq:prop_r1} and then proving \eqref{eq:prop_r1r2}. }, we show that
    for every $1 \le k' \le k/2$ and $\vecp{b} \in \Rxhom[d]^2$,
    \begin{align*}
      \textstyle
      \dotp{\Aup(\vecp{b}) x_1^{2k'}/\prod_{i=1}^{k'}r_i, \sig(\vecp{u}) - p} = 0.
    \end{align*}
    This is because each $r_i$ is a divisor of $\sig(\vecp{u'})$ and
    $\prod_{i=1}^{k'}r_i$ divides $\gc$. From here we can finish our proof by
    following the same steps as in the proof of \Cref{prop:second_order_property}.
  \end{proof}
\end{proposition}

With \Cref{prop:second_order_property_factor} we can prove \Cref{thm:main}, by
showing that every $p \in \Sxhom[2d]$ has the required decomposition.

\begin{proof}[Proof of \Cref{thm:main}]
  Since $\dotp{\Aupp(\vecp{b}) g x_1^{k}, \sig(\vecp{u}) - p} = 0$ for all
  $\vecp{b} \in \Rxhom[d]^2$ and $u_1', u_2'$ are coprime,
  \Cref{lem:coprime_kernel} implies that
  $\dotp{wgx_1^{k} , \sig(\vecp{u}) - p} = 0$ for all $w \in \R[x]$. Since
  $\sig(\vecp{u}')$ has no real roots (\Cref{prop:no-real-roots}), it is coprime
  with $x_1^k$. As $\sig(\vecp{u}')$ is coprime with $g$, it is also coprime
  with $gx_1^k$. Thus \Cref{lem:hgroup} tells us that there exists a sum of
  squares polynomial $s$ such that $p \equiv s \sig(\vecp{u}')
  \pmod{gx_1^k}$. Since $s \sig(\vecp{u}') \in \cone(\sig(\ker(\Aup)))$ by
  \eqref{eq:conesigker}, we are done.
\end{proof}

\section{Geometric Interpretation and Certificates}
\label{sec:cert}
In this section, we provide a geometric interpretation of our proof of
\Cref{thm:main}, which allows us to turn the proof into a certificate. In order
to prove that there is no spurious second-order critical points when minimizing
$f_p(\vecp{u}) = \norm{\sig(\vecp{u}) - p}^2$, we have to show that for all
$\vecp{u} \in \R[x]^r$ and for all $p \in \Sigma[x]$, $\nabla f_p(\vecp{u}) = 0$
and $\nabla^2 f_p(\vecp{u}) \succeq 0$ implies that $f_p(\vecp{u}) = 0$ and
$p = \sig(\vecp{u})$. One way to tackle this problem is to fix $p$ then
characterize the set of $\vecp{u}$ satisfying the second-order critical point
conditions. This is the approach taken by \cite{bhojanapalli2018smoothed} and
related works, where they used an argument based on the dimension of the
subspace generated by the constraints of the SDP. However the SOCP conditions
are nonconvex in $\vecp{u}$. In order to do better than a dimension-counting
argument, our proof takes a different approach. If we fix $\vecp{u}$, the set of
all $p$ satisfying the gradient condition \eqref{eq:gradA} is an affine
subspace, whereas the set of all $p \in \Sigma[x]$ satisfying the Hessian
condition \eqref{eq:hessA} is a convex semidefinite-representable set. We need
to show that these two sets intersect at only one point, $p = \sig(\vecp{u})$
(see \Cref{fig:geometric}).

\begin{figure}[ht]
  \centering
  \def\svgwidth{\columnwidth/2}
\begingroup%
  \makeatletter%
  \providecommand\color[2][]{%
    \errmessage{(Inkscape) Color is used for the text in Inkscape, but the package 'color.sty' is not loaded}%
    \renewcommand\color[2][]{}%
  }%
  \providecommand\transparent[1]{%
    \errmessage{(Inkscape) Transparency is used (non-zero) for the text in Inkscape, but the package 'transparent.sty' is not loaded}%
    \renewcommand\transparent[1]{}%
  }%
  \providecommand\rotatebox[2]{#2}%
  \newcommand*\fsize{\dimexpr\f@size pt\relax}%
  \newcommand*\lineheight[1]{\fontsize{\fsize}{#1\fsize}\selectfont}%
  \ifx\svgwidth\undefined%
    \setlength{\unitlength}{197.11842935bp}%
    \ifx\svgscale\undefined%
      \relax%
    \else%
      \setlength{\unitlength}{\unitlength * \real{\svgscale}}%
    \fi%
  \else%
    \setlength{\unitlength}{\svgwidth}%
  \fi%
  \global\let\svgwidth\undefined%
  \global\let\svgscale\undefined%
  \makeatother%
  \begin{picture}(1,0.62066998)%
    \lineheight{1}%
    \setlength\tabcolsep{0pt}%
    \put(0,0){\includegraphics[width=\unitlength,page=1]{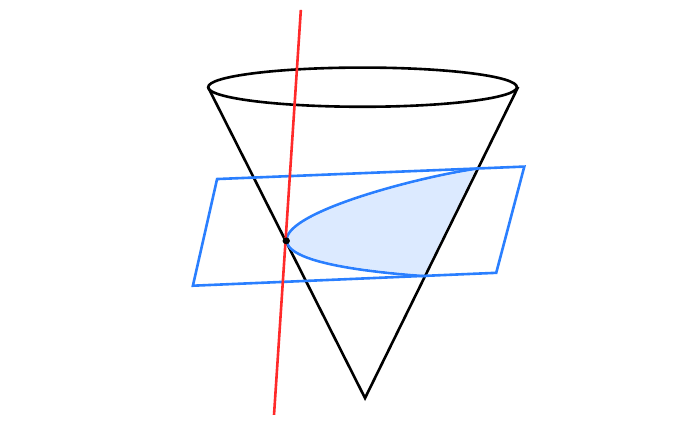}}%
    \put(0.77091477,0.27232041){\color[rgb]{0.16470588,0.49803922,1}\makebox(0,0)[lt]{\smash{\begin{tabular}[t]{l}$\braces{p \in \Sigma[x] \mid \nabla^2 f_p(\vecp{u}) \succeq 0}$\end{tabular}}}}%
    \put(-0.1065204,0.07901583){\color[rgb]{1,0.16470588,0.16470588}\makebox(0,0)[lt]{\smash{\begin{tabular}[t]{l}$\braces{p \in \R[x] \mid \nabla f_p(\vecp{u}) = 0}$\end{tabular}}}}%
    \put(0.01438042,0.31333519){\color[rgb]{0,0,0}\makebox(0,0)[lt]{\smash{\begin{tabular}[t]{l}$p = \sigma(\vecp{u})$\end{tabular}}}}%
    \put(0,0){\includegraphics[width=\unitlength,page=2]{geometric.pdf}}%
  \end{picture}%
\endgroup%

  \caption{\label{fig:geometric} The geometric interpretation}
\end{figure}

Our proof can be interpreted as constructing a certificate to show that these
two sets only intersect at one point. This is true if and only if the following
optimization problem has a zero optimal objective value:
\begin{align} \label{eq:lagrangian}
  \max_{p \in \Sigma[x]} \min_{Q \succeq 0,\, \vecp{\lambda} \in \R[x]^r}
  \norm{\sig(\vecp{u}) - p}^2 + \nabla f_p(\vecp{u})(\vecp{\lambda})
  + \dotp{Q, \nabla^2 f_p(\vecp{u})}.
\end{align}
Expanding the gradient and Hessian, we get
\begin{align*}
  \nabla f_p(\vecp{u})(\vecp{\lambda}) &= \dotp{\Aup(\vecp{\lambda}), \sig(\vecp{u}) - p} \\
  \dotp{Q, \nabla^2 f_p(\vecp{u})} &= \sum_i \dotp{\sig(\vecp{v}^{(i)}), \sig(\vecp{u}) - p}
                                     + 2 \norm{\Aup(\vecp{v}^{(i)})}^2.
\end{align*}
where $Q = \sum_i \vecp{v}^{(i)} {\vecp{v}^{(i)}}^\top$. If for every
$p \in \Sigma[x]$ we can find $\vecp{\lambda}$ and $Q$ such that
\begin{align}
  \label{eq:lambdaQ}
  \nabla f_p(\vecp{u})(\vecp{\lambda}) + \dotp{Q, \nabla^2 f_p(\vecp{u})}
  = -\norm{\sig(\vecp{u}) - p}^2,
\end{align}
then the objective of \eqref{eq:lagrangian} is at most 0 and cannot be positive,
showing that $p = \sig(\vecp{u})$ is the only point satisfying the gradient
and Hessian conditions. Since $\Aup(\vecp{u}) = \sig(\vecp{u})$, this is
equivalent to finding $\vecp{\lambda}$ and $Q$ such that
$\nabla f_p(\vecp{u})(\vecp{\lambda}) + \dotp{Q, \nabla^2 f_p(\vecp{u})} =
\dotp{p, \sig(\vecp{u}) - p}$.

\subsection{Warmup}
As a warmup, we construct such a certificate if $h=1$ in the decomposition of
$\vecp{u}$ in \Cref{prop:decomp}. Recall that in this case $g = \gcd(u_1, u_2)$,
$u_1 = gu_1', u_2 = gu_2'$ and $\sig(\vecp{u}')$ is coprime with
$g \in \Rxhom[m]$. Therefore, by \Cref{lem:hgroup}, there exists
$s \in \Sxhom[2m]$ such that $p \equiv s \sig(\vecp{u}') \pmod{g}$. Let
\begin{align*}
  Q & = s
      \begin{bmatrix}
    {u_2'}^2 & -u_1'u_2'\\
    -u_1'u_2' & {u_1'}^2
  \end{bmatrix}.
\end{align*}
As both $\sig(\vecp{u})$ and $p - s \sig(\vecp{u}')$ are divisible by $g$,
\Cref{lem:coprime_kernel} implies that there exists $\vecp{\lambda}$ such that:
\begin{align*}
  \Aup(\vecp{\lambda}) & = -\sig(\vecp{u}) + p - s \sig(\vecp{u}').
\end{align*}
These values of $\vecp{\lambda}$ and $Q$ give
\begin{align*}
  \nabla f_p(\vecp{u})(\vecp{\lambda})
  & = -\norm{\sig(\vecp{u}) - p}^2
    - \dotp{s\sig(\vecp{u}'), \sig(\vecp{u}) - p}\\
  \dotp{Q, \nabla^2 f_p(\vecp{u})}
  &= \dotp{s \sig(\vecp{u}'), \sig(\vecp{u}) - p},
\end{align*}
hence taking the sum we have the identity \eqref{eq:lambdaQ}.

\subsection{Certificate}
\label{sec:cert-general}
Now we can present the proof of \Cref{thm:main} in the form of a
certificate. First we decompose $\vecp{u}\in \Rxhom[d]^2$ as in
\Cref{prop:decomp}. Since $g$ is coprime with $\sig(\vecp{u}')$ and
$\sig(\vecp{u}')$ has no real roots, $g x_1^k$ is also coprime with
$\sig(\vecp{u}')$. Then by \Cref{lem:hgroup}, there exists
$s \in \Sxhom[2(m+k)]$ such that $p \equiv s \sig(\vecp{u}') \pmod{g x_1^k}$.
Next we apply \Cref{lem:coprime_kernel} to find $\vecp{b}^0 \in \Rxhom[d]^2$ so
that:
\begin{align*}
  2 \frac{x_1^k}{h} \Aup(\vecp{b}^0) = 2g x_1^k\Aupp(\vecp{b}^0) = p - s \sig(\vecp{u}').
\end{align*}
As in the proof of \Cref{prop:second_order_property_factor}, we write
$h = \prod_{i=1}^{k/2} r_i$.  For every $1 \le j \le k/2$, since $r_j$ divides
$\sig(\vecp{u'})$, by \Cref{lem:coprime_kernel} there exists
$\vecp{b}^j \in \Rxhom[d]^2$ so that
\begin{align*}
  \textstyle
  2 g x_1^{k-2j} \Aupp(\vecp{b}^j) \prod_{i=1}^j r_i
  = -  g^2 x_1^{2k-4(j-1)} \sig(\vecp{u'}) \prod_{i=1}^{j-1} r_i^2.
\end{align*}
Given any $\vecp{a} = (a_1, a_2) \in \Rxhom[d]^2$, we define
$\vecp{\bar{a}} := (a_2, -a_1)$. Given a parameter $\eta \in \R$, for every
$0 \le j \le k/2$ let $\eta_j = \eta^{3^j}$ and
\begin{align*}
  \textstyle
  \vecp{v}^j = \eta_j^{3/2} g x_1^{k-2j} \vecp{\bar{u}'} \prod_{i=1}^j r_i
  + \eta_j^{-1/2} \vecp{\bar{b}}^j.
\end{align*}
Then define
\begin{align*}
  Q &= s \vecp{\bar{u}'} \vecp{\bar{u}'}^\top + \frac{1}{\eta}\sum_{j=0}^{k/2} \vecp{v}^j {\vecp{v}^j}^\top \\
  \vecp{\lambda} &= -(1 + \eta^{-1}\eta_{k/2+1}) \vecp{u}.
\end{align*}
Since $\Aup(\vecp{\bar{u}'}) = 0$, $\sig(\vecp{\bar{u}'}) = \sig(\vecp{u'})$,
$\sig(\vecp{\bar{b}}^j) = \sig(\vecp{b}^j)$,
$\mathcal{A}_{\vecp{\bar{u}'}}(\vecp{\bar{b}}^j) = \Aupp(\vecp{b}^j)$ and
$\eta_{j+1} = \eta_j^3$, we have
\begin{align*}
  \frac{1}{\eta}\sum_{j=0}^{k/2}\sig(\vecp{v}^j)
  &= \frac{1}{\eta}\sum_{j=0}^{k/2} \paren{ \textstyle
    \eta_j^3 g^2 x_1^{2k-4j} \sig(\vecp{u'}) \prod_{i=1}^j r_i^2
    + 2\eta_j g x_1^{k-2j}  \Aupp(\vecp{b}^j) \prod_{i=1}^j r_i
    + \eta_j^{-1} \sig(\vecp{b}^j)} \\
  &= p - s\sig(\vecp{u'}) + \eta^{-1}\eta_{k/2+1}\sig(\vecp{u})
    + \eta^{-1}\sum_{j=0}^{k/2} \eta_j^{-1}\sig(\vecp{b}^j), \\
  \Aup(\vecp{v}^j)
  &= \eta_j^{-1/2} \Aup(\vecp{\bar{b}}^j),\\
  \dotp{Q, \nabla^2 f_p(\vecp{u})}
  &= \dotp{p + \eta^{-1}\eta_{k/2+1}\sig(\vecp{u})
    + \eta^{-1}\sum_{j=0}^{k/2} \eta_j^{-1}\sig(\vecp{b}^j), \sig(\vecp{u}) - p}
    + \sum_{j=0}^{k/2}2\eta_j^{-1} \eta^{-1} \norm{\Aup(\vecp{\bar{b}}^j)}^2, \\
  \nabla f_p(\vecp{u})(\vecp{\lambda}) &= -(1 + \eta^{-1}\eta_{k/2+1})\dotp{\sig(\vecp{u}) , \sig(\vecp{u}) - p}.
\end{align*}
So we have proven the identity
\begin{align} \label{eq:cert-general}
  \nabla f_p(\vecp{u})(\vecp{\lambda}) + \dotp{Q, \nabla^2 f_p(\vecp{u})}
  & =
    -\norm{\sig(\vecp{u}) - p}^2  +
    \sum_{j=0}^{k/2} \frac{1}{\eta \eta_j}
    \paren{\dotp{\sig(\vecp{b}^j), \sig(\vecp{u}) - p} + 2\norm{\Aup(\vecp{\bar{b}}^j)}^2}.
\end{align}
This implies that for every $\eta > 0$ and every $\vecp{u}$ that satisfies
$\nabla f_p(\vecp{u}) = 0$ and $\nabla^2 f_p(\vecp{u}) \succeq 0$,
\begin{align*}
  \norm{\sig(\vecp{u}) - p}^2 \le
    \sum_{j=0}^{k/2} \eta^{-(3^j+1)}
    \paren{\dotp{\sig(\vecp{b}^j), \sig(\vecp{u}) - p} + 2\norm{\Aup(\vecp{\bar{b}}^j)}^2}.
\end{align*}
Since $\vecp{b}^j$ does not depend on $\eta$, we can make the right hand side
arbitrarily small by taking the limit $\eta \to \infty$. Thus we can conclude
that $\norm{\sig(\vecp{u}) - p} = 0$.

\section{Extensions and Generalizations}
\label{sec:extensions}
The certificate interpretation discussed in the previous section allows us to
generalize \Cref{thm:main} to other settings, such as projecting onto the sum of
squares cone, certifying nonnegativity on intervals and imposing linear
constraints on coefficients of univariate sum of squares polynomials.

\subsection{Projection Onto the Sum of Squares Cone}
A natural question to consider is what happens to the optimization landscape of
$f_p(\vecp{u})$ when $p$ cannot be expressed as a sum of squares. In this case
the objective $f_p(\vecp{u})$ can never be zero, but we show that all SOCPs have
the same objective value, which is the projection of $p$ to the sum of squares
cone.

\begin{corollary} \label{cor:projection} For all $\vecp{u} \in \Rxhom[d]^2$
  where $\nabla f_p(\vecp{u}) = 0$ and $\nabla^2 f_p (\vecp{u}) \succeq 0$,
  $\sig(\vecp{u})$ is the projection of $p$ to the sum of squares cone with
  respect to the inner product used to define $f_p$. In other words,
  $f_p(\vecp{u}) = \norm{\sig(\vecp{u})-p}^2 \le \norm{q - p}^2$ for all
  $q \in \Sxhom[2d]$.
\end{corollary}
\begin{proof}
  \Cref{cor:projection} can be proved by a simple modification of the
  certificate \eqref{eq:cert-general}. Although $p$ is no longer a sum of
  squares, we can use \eqref{eq:cert-general} to show that for all
  $q \in \Sxhom[2d]$,
  \begin{align*}
    \dotp{\sig(u) - q,\sig(u) - p} \le 0.
  \end{align*}
  This is exactly the variational characterization of projection onto the convex
  cone $\Sxhom[2d]$.
\end{proof}

\subsection{Certifying Nonnegativity on Intervals}
Suppose we wish to certify that a univariate polynomial $p(x)$ is nonnegative in
a union of intervals $I = \bigcup_{i=1}^m I_i$ where
$I_i = \{x \in \R \mid \alpha_i \le x \le \beta_i\}$. This can be accomplished
by finding a decomposition
\begin{align*}
  p(x) = \sum_i^m a_i(x) q_i(x),
\end{align*}
where $a_i(x)$ are fixed polynomials depending on the intervals $I_i$ and
$q_i(x)$ are sum of squares polynomials (see, e.g., \cite[Theorem
3.72]{blekherman_semidefinite_2013}). This objective can also be written as a
nonconvex optimization problem by the decomposition
$q_i(x) = \sum_{j=1}^r u_{ij}(x)^2 = \sig(\vecp{u}_i)$. If we let
\begin{align*}
  s(\vecp{u})(x) = \sum_{i=1}^m a_i(x) \sig(\vecp{u}_i),
\end{align*}
then the objective $f^I_p(\vecp{u})$ and its gradient and Hessian can be written as
\begin{align*}
  f^I_p(\vecp{u})
  &= \norm{\sum_{i=1}^m a_i \sig(\vecp{u}_i) - p}^2 = \norm{s(\vecp{u})-p}^2,\\
  \nabla f^I_p(\vecp{u})(\vecp{v})
  &= \dotp{\sum_{i=1}^m a_i \mathcal{A}_{\vecp{u}_i}(\vecp{v}_i), s(\vecp{u})-p}, \\
  \nabla^2 f^I_p(\vecp{u})(\vecp{v},\vecp{v})
  &= \dotp{\sum_{i=1}^m a_i \sig(\vecp{v}_i), s(\vecp{u})-p}
    + 2\norm{\sum_{i=1}^m a_i \mathcal{A}_{\vecp{u}_i}(\vecp{v}_i)}^2.
\end{align*}
\begin{corollary} \label{cor:intervals} Suppose $r \ge 2$ and we are given
  $p = \sum_{i=1}^m a_i q_i \in \Rxhom[2d+k]$ where $a_i \in \Rxhom[k]$,
  $q_i \in \Sxhom[2d]$. For all $\vecp{u} \in \Rxhom[d]^{m \times r}$ such that
  $\nabla f^I_p(\vecp{u}) = 0$ and $\nabla f^I_p(\vecp{u}) \succeq 0$,
  $f^I_p(\vecp{u}) = 0$.
\end{corollary}
\begin{proof}
  We can prove this by constructing a certificate of the form
  \eqref{eq:cert-general}. For each $i$ we can choose $\vecp{v}_j = 0$ for all
  $j \ne i$, then follow the reasoning in \Cref{sec:cert-general} to find
  $\vecp{\lambda_i}$ and $Q_i$ such that for all $\eta_i > 0$,
  \begin{align} \label{eq:cert-multivariate}
    \nabla f^I_p(\vecp{u})(\vecp{\lambda_i}) + \dotp{Q_i, \nabla^2 f^I_p(\vecp{u})}
    = \dotp{a_iq_i, s(\vecp{u})-p} + C_i,
  \end{align}
  where $C_i$ is a value that can be made arbitrarily small by taking a
  limit. We then sum \eqref{eq:cert-multivariate} for all $i$, along with the
  equality
  $\nabla f^I_p(\vecp{u})(-\vecp{u}) = -\dotp{s(\vecp{u}), s(\vecp{u})-p}$ to
  get $\norm{s(\vecp{u})-p}^2 \le \sum_i C_i$, which implies that
  $f^I_p(\vecp{u}) = \norm{s(\vecp{u})-p}^2 = 0$.
\end{proof}

\subsection{Sum of Squares Optimization}
More generally, we can consider the problem of finding a feasible point in the
intersection of the cone $\Sx[2d]$ with any affine subspace. This allows us to
solve sum of squares optimization problems involving univariate polynomials. Let
$\mathcal{B} : \Rxhom[2d] \rightarrow \R^m$ be a linear map. Given $b \in \R^m$,
we want to find $p \in \Sxhom[2d]$ so that $\mathcal{B}(p) = b$. This is
equivalent to minimizing the quadratic-penalized problem
\begin{align} \label{eq:seminorm}
  f_{\mathcal{B}}(\vecp{u}) = \norm{\mathcal{B}(\sig(\vecp{u})) - b}^2.
\end{align}
\begin{corollary} \label{cor:sos-opt}
  Suppose there exists $p \in \Sxhom[2d]$ such that $\mathcal{B}(p) = b$. Then
  $\nabla f_{\mathcal{B}}(\vecp{u}) = 0$ and
  $\nabla^2 f_{\mathcal{B}}(\vecp{u}) \succeq 0$ implies that
  $f_{\mathcal{B}}(\vecp{u}) = 0$.
\end{corollary}
\begin{proof}
  The gradient and Hessian of the objective \eqref{eq:seminorm} can be written
  as
  \begin{align*}
    \frac{1}{4} \nabla f_{\mathcal{B}}(\vecp{u})(\vecp{v})
    &= \dotp{\mathcal{B}(\Aup(\vecp{v})), \mathcal{B}(\sig(\vecp{u}) - p)} \\
    \frac{1}{4} \nabla^2 f_{\mathcal{B}}(\vecp{u})(\vecp{v}, \vecp{v})
    &= \dotp{\mathcal{B}(\sig(\vecp{v})), \mathcal{B}(\sig(\vecp{u}) - p)}
      + 2\norm{\mathcal{B}(\Aup(\vecp{v}))}^2.
  \end{align*}
  Thus by the linearity of $\mathcal{B}$ we can use the same construction as in
  the certificate \eqref{eq:cert-general} to show that
  $f_{\mathcal{B}}(\vecp{u}) = 0$.
\end{proof}

\section{Implementation and Experiments}
\label{sec:experiments}
In this section we describe an efficient implementation of finding a sum of
squares decomposition of trigonometric polynomials. A trigonometric polynomial
of degree-$d$ is defined by $2d+1$ coefficients and has the form
\begin{align*}
  p(t) = a_0 + \sum_{k=1}^d (a_k \cos(kt) + a_{-k} \sin(kt)).
\end{align*}
By the substitution $\cos(t) = \frac{1-x^2}{1+x^2}$ and
$\sin(t) = \frac{2x}{1+x^2}$, $p(x)$ becomes a rational function with the
denominator a power of $1+x^2$ and the numerator a degree-$2d$ polynomial in
$x$. Thus certifying the nonnegativity of the numerator is equivalent to
certifying the nonnegativity of $p(t)$. By this correspondence the result of
\Cref{thm:main} also applies to trigonometric polynomials.

Since the proof of \Cref{thm:main} does not depend on the norm used for
$f_p(\vecp{u}) = \norm{p - \sum_i u_i^2}^2$, we can choose one most suitable for
the gradient computation. For the rest of this section we assume that $d$ is
even for simplicity of notation; a similar decomposition exists for odd $d$ by
choosing ``half-angles'' (see \cite{Lofbergcoefficientssamplesnew2004} for more
details). We then choose the inner product defined by evaluation at $2d+1$
points on the circle,
\begin{align*}
  \dotp{p, q} = \frac{1}{2d+1} \sum_{k=1}^{2d+1} p(x_k) q(x_k), \quad x_k = \frac{2k\pi}{2d+1}.
\end{align*}
Since a trigonometric polynomial of degree-$d$ is uniquely defined by evaluation
on $2d+1$ unique points, $\norm{p(x)}^2 = 0$ if and only if $p$ is identically
zero.

Let $U \in \R^{(d+1) \times r}$ be a matrix with column $U_i$ representing the
coefficients of $u_i(x)$, and $B \in \R^{(d+1) \times (2d+1)}$ be the evaluation
map on $2d+1$ points with columns
\begin{align*}
  B_k = \bmat{1 & \cos(x_k) & \cdots & \cos(\frac{d}{2} x_k) &
              \sin(x_k) & \cdots & \sin(\frac{d}{2} x_k)}^\top,
\end{align*}
so that $B_k^\top U_i = u_i(x_k)$. Let $\bar{p}$ be the vector of coefficients of
$p(x)$, so that $B_k^\top \bar{p} = p(x_k)$. Then we can write
\begin{align*}
  f_p(U) &= \frac{1}{2d+1} \sum_{k=1}^{2d+1} \paren{\norm{U^\top B_k}^2 - p(x_k)}^2 \\
  \nabla f_p(U) &= \frac{4}{2d+1} U^\top B \Diag\paren{\norm{U^\top B_k}^2 - p(x_k)} B^\top,
\end{align*}
where $\Diag\paren{\norm{U^\top B_k}^2 - p(x_k)}$ is a diagonal matrix with
$\norm{U^\top B_k}^2 - p(x_k)$ as the $k$-th diagonal entry. Since matrix-vector
multiplication by $B$ is a equivalent to a discrete Fourier transform,
$\nabla f_p(U)$ can be computed in $O(r d \log d)$ time using the
FFT. \Cref{thm:main} shows that spurious local minima do not exist when
$r \ge 2$, so we can pick $r$ to be a constant and obtain a near-linear
iteration complexity. This is in contrast to other SDP-based algorithms and
custom interior point methods for solving this problem, which run into
computational difficulties even when $2d = \text{10,000}$.

\begin{table}[ht]
  \centering
  \subfloat[Varying degree, $r=4$]{\begin{tabular}[t]{rr@{\,}lr@{\,}l}
\toprule
Degree & \multicolumn{2}{c}{Time (s)} & \multicolumn{2}{c}{Iterations} \\
\midrule

2,000 & 2 & (1 -- 2) & 340 & (306 -- 384) \\
10,000 & 6 & (5 -- 6) & 530 & (497 -- 592) \\
20,000 & 9 & (8 -- 10) & 632 & (587 -- 695) \\
100,000 & 53 & (46 -- 59) & 1126 & (980 -- 1248) \\
200,000 & 160 & (139 -- 174) & 1375 & (1212 -- 1532) \\
1,000,000 & 1461 & (1212 -- 1532) & 2303 & (1934 -- 2437) \\
\bottomrule
\end{tabular}

 \label{table:timing}}
  \hfill
  \subfloat[Varying $r$, degree $2d = 10,000$]{\begin{tabular}[t]{rccc}
\toprule
$r$ & Time & Iters & FFT Calls \\
\midrule

2 & 50 & 4124 & 20618 \\
3 & 9 & 896 & 6272 \\
4 & 6 & 530 & 4774 \\
5 & 5 & 446 & 4900 \\
6 & 5 & 396 & 5142 \\
7 & 5 & 374 & 5618 \\
\bottomrule
\end{tabular}

 \label{table:r}}
  \caption{Time and iterations to convergence for sum of squares decomposition
    of random nonnegative trigonometric polynomials. All values are median of
    50 runs (with range based on 25th and 75th percentile).
    \Cref{table:timing} fixes the rank $r$ and and varies the polynomial degree,
    whereas \Cref{table:r} fixes the polynomial degree and varies the rank $r$. }
\end{table}
We implemented our algorithm for finding the sum of squares decomposition of
trigonometric polynomials in Julia \footnote{The code and data required to
  reproduce the results in this section can be found in
  \cite{legat2023-low-rank-paper-code}.}, using the
\texttt{FFTW.jl}~\cite{frigo2005design} package for FFTs to compute
$\nabla f_p(U)$ and the \texttt{NLopt.jl}~\cite{johnson2014nlopt} package to
minimize $f_p(U)$ using a first-order algorithm (L-BFGS). We performed the
timing experiments on Intel Xeon Platinum 8260 processors, allocating at least
$r+1$ cores to each run, using polynomials of degree-$2d$ ranging from 2,000 to
1,000,000. The test polynomials are generated with coefficients drawn from a
standard normal distribution, with a constant coefficient added so that they all
have a small positive minimum value. $U$ is initialized with a small random
value; its magnitude depends on the size of the problem. The algorithm is
terminated when the relative error for each entry of $U$ is on the order of
$10^{-7}$. Although $r=2$ is sufficient, the results in \Cref{table:r} shows
that $r=4$ minimizes the total computational cost, as measured by the total
number of FFT calls (by far the most expensive operation) needed for
convergence. In addition, since the matrix-vector products can be easily
parallelized across multiple threads, increasing $r$ does not incur a
significant per-iteration cost if sufficient threads are used. Thus we choose
$r=4$ for our large-scale experiments in
\Cref{table:timing}. \Cref{fig:iterations} plots the convergence rate of 20
instances, and we can see that they achieve a linear convergence rate. This is
in contrast to grid-based methods \cite{wu1999fir,dumitrescu2007positive} which
scales sublinearly in accuracy.

\begin{figure}[ht]
  \centering
  \includegraphics[scale=0.5]{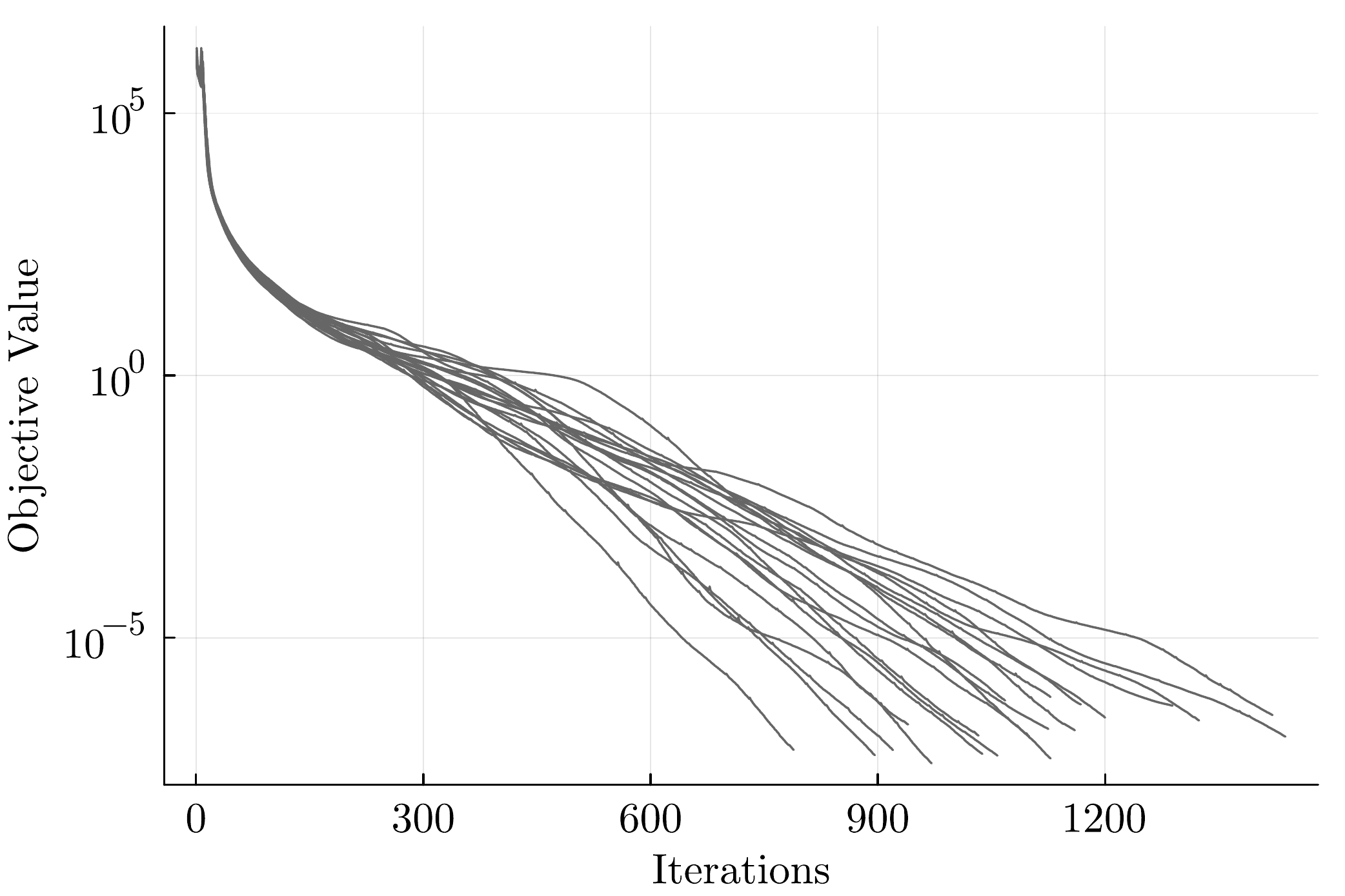}
  \caption{\label{fig:iterations} Value of $f_p(\vecp{u})$ against iterations of
    L-BFGS for computing the sum of squares decomposition of 20 random
    nonnegative trigonometric polynomials ($2d = 100,000$), showing a linear
    convergence rate. }
\end{figure}

\section{Conclusion}
When does it make sense to solve nonconvex formulations of convex problems? In
this paper we addressed this question for sum of squares decomposition and
optimization of univariate polynomials, showing that solving the nonconvex
formulation can provide a large computational speedup while still maintaining
provable guarantees on the convergence to the global optima. Key to our approach
is retaining polynomial structure in the nonconvex formulation. This enables us
to use algebraic methods to construct a certificate showing that all SOCPs are
global minima.

Our approach for finding sum of squares decompositions generalizes to
multivariate polynomials, although we do not have guarantees for the rank needed
to exclude spurious second-order critical points. On the other hand, results for
low-rank matrix factorization tell us that this rank is equal to the Pythagoras
number for quadratic forms. Thus we conjecture that a version of \Cref{thm:main}
is true for ternary quartics and matrix polynomials, cases where nonnegativity
is equivalent to the existence of a sum of squares decomposition.  Some
similarities between our conditions and a new characterization of theses cases
in terms of varieties of minimal degree~\cite{blekherman2019low} also suggest
that \Cref{thm:main} could be generalized to these cases.  In particular, the
case where the syzygy module only contains the \emph{Koszul
  syzygies}~\cite[p.~581]{cox2015groebner} could generalize the coprime case
studied in \Cref{sec:g=1_h=1}.

Another direction for future work is to apply our methods to other structured
semidefinite programs or polynomial-valued objectives such as symmetric tensor
decomposition.

\section{Acknowledgments}
We thank the two anonymous referees for their valuable feedback and comments.
B. Legat was supported by a BAEF Postdoctoral Fellowship and the NSF grant
OAC-1835443. C. Yuan was supported by the NSF grant CCF-1565235 and AFOSR grant
FA8750-19-2-1000.

\bibliographystyle{siamplain}
\bibliography{../MyLibraryTEX}
\end{document}